\documentclass[10pt,a4paper]{amsart}%

\usepackage{MyCommA}
\usepackage[final]{pdfpages}
\newcommand{\RHom}{\operatorname{RHom}}

\newcommand{\simpAr}[2][r]{%
\ar@{}[#1]|-*[@]_{#2}%
}
\renewcommand{\colorlinks}{true}
\renewcommand{\linkcolor}{lblue}
\renewcommand{\citecolor}{lblue}
\renewcommand{\urlcolor}{dblue}
\renewcommand{\linkbordercolor}{red}
\renewcommand{\citebordercolor}{green}
\renewcommand{\urlbordercolor}{cyan}
\hypersetup{colorlinks=\colorlinks,linkbordercolor=\linkbordercolor, linkcolor=\linkcolor, citecolor=\citecolor, urlcolor=\urlcolor,urlbordercolor=\urlbordercolor,citebordercolor=\citebordercolor}%
\newcommand{\CH}{\mathfrak{X}}

\renewcommand{\bfM}{\mathbf{M}}
\renewcommand{\bfG}{\mathbf{G}}
\renewcommand{\bfH}{\mathbf{H}}

\newcommand{\Eitan}[1]{{{#1}}}
\newcommand{\EitanB}[1]{{{#1}}}
\newcommand{\Rami}[1]{{{#1}}}
\newcommand{\RamiD}[1]{{{#1}}}
\newcommand{\RamiE}[1]{{{#1}}}
\newcommand{\RamiF}[1]{{{#1}}}
\newcommand{\RamiG}[1]{{{#1}}}
\newcommand{\RamiH}[1]{{{#1}}}

\newcommand{\EitanG}[1]{{{#1}}}
\newcommand{\EitanH}[1]{{{#1}}}
\newcommand{\RamiK}[1]{{{#1}}}
\newcommand{\NextVer}[1]{}

\usepackage{amsmath}

\newcommand{\zerodel}{.\kern-\nulldelimiterspace}

\begin{document}

\author{Avraham Aizenbud}
\address{Avraham Aizenbud,
\RamiE{The Incumbent of Dr. A. Edward Friedmann Career Development Chair in Mathematics,}
Faculty of Mathematical Sciences,
Weizmann Institute of Science,
76100 Rehovot, Israel }
\email{aizenr@gmail.com}
\urladdr{http://www.wisdom.weizmann.ac.il/~aizenr/}
\author{Eitan Sayag}
\address{Eitan Sayag, Department of Mathematics, Ben-Gurion University of the Negev, ISRAEL}
\email{sayage@math.bgu.ac.il}
\urladdr{http://www.math.bgu.ac.il/~sayage/}
\title{On the Cohen-Macaulay Property of Spherical Varieties}
\date{\today}

%
%
%
%
%
%
%
%

\title[Homological multiplicities]{Homological multiplicities in representation theory of $p$-adic groups}
\begin{abstract}
We study homological multiplicities of spherical varieties of reductive \RamiK{group $G$} over \RamiK{a} $p$-adic field \RamiK{$F$}. Based on Bernstein\RamiE{'s} decomposition of the category of smooth representations of a $p$-adic group, we introduce a sheaf that measures these multiplicities.

We show that these multiplicities are finite \EitanB{whenever the usual mutliplicities are finite}, in particular this holds for symmetric \RamiD{varieties}, conjectured for \RamiD{all} spherical \RamiD{varieties} and known for \RamiD{a large class of}  spherical \RamiD{varieties}. Furthermore, we show that the Euler-Poincar\'e characteristic is constant in families induced from \EitanH{admissible} 
representations
of a Levi $M.$
In the case when $M=G$ we compute these multiplicities more explicitly.

\end{abstract}
\maketitle
\tableofcontents

\section{Introduction}

Let $\bf{G}$ be a \RamiF{connected} reductive group defined over a non-Archimedean local field $F$. 
Let $X$ be \RamiE{an} $F$-spherical  transitive $G=\bfG(F)$-space. That is, a $G$-transitive space which \RamiE{has} finitely many $P_0$ orbits,
where $P_0$ is a minimal parabolic
subgroup of $G$. Let $\Sc(X)$ be the space of Schwartz (i.e. locally constant compactly supported) functions on $X$. Harmonic analysis on $X$ requires a deep understanding of the multiplicity spaces  $\Hom_{G}(\Sc(X),\pi)$, where $\pi$ is \RamiE{an} admissible representation of $G.$ A considerable amount of research was dedicated to the determination of the dimension of these multiplicity spaces (e.g. \cite{GK,Sha,JR,Pra_m,vD,AGRS,AGS,HM,AG_HC,AG_AMOT,SZ,GGP,OS}) and to \RamiE{the} construction of bas\RamiE{es} to these spaces (e.g. \cite{vB,OS_e,GSS}).

Assume that for any admissible representation $\pi$ of $G$, the multiplicity \Eitan{$$m_{X}(\pi):=\dim\RamiK{\Hom_G}(\Sc(X),\pi)$$ is finite.} This is known to be true in a number of cases (see \cite{Del,SV}), including the case of symmetric spaces, and conjectured to be true for any $F$-spherical space.

In this paper we study the homological version of those multiplicities:
$$m^{i}_{X}(\pi)=\dim\RamiK{\Ext^{i}_G}(\Sc(X),\pi), \RamiG{\text{ for } i\geq 0}.$$

The first result of this paper is the following:
\begin{introtheorem}[See Theorem \ref{thm.f.mul} {for  \RamiE{a} more general result}]\label{thm:intro.f.mul}
Let $\pi$ be an admissible representation of $G.$ Under the assumption above, the homological multiplicities
  $m^{i}_{X}(\pi)$
are finite for every $\RamiG{i\geq 0}.$

\end{introtheorem}
Special cases of this result were studied in \cite[\S5]{Pra} using a different method.

The category $\mathcal{M}(G)$ of smooth representations of $G$ is of finite homological dimension (\cite[Theorem 29]{BerLec}).
Thus, we can define
$$e_{X}(\pi):= EP_{G}(\Sc(X),\pi):=\sum (-1)^i \dim\RamiK{\Ext^i_G}(\Sc(X),\pi).$$
More generally we will consider
$$p_{X}(\pi):= P_{G}(\Sc(X),\pi):=\sum  \dim\RamiK{\Ext^i_G}(\Sc(X),\pi) t^i\in \Z[t].$$
The second result of this paper is the following:
\begin{introtheorem}[See Theorem \ref{thm:ep} \Rami{for more general result}]\label{thm:intro.ep}
Let $\bfM\subset \bfG$ be a Levi subgroup of $\bfG$. Let $\rho$ be \RamiG{an irreducible cuspidal} representation of $M=\bfM(F).$ For an unramified character $\chi$ of $M$, let $\pi_\chi:=i_M^G(\rho \cdot \chi)$ be the normalized parabolic induction. 
 Then $e_{X}(\pi_\chi)$ is independent of $\chi$.
\end{introtheorem}


In the case \RamiK{where} $\bfM=\bfG$ we also prove:
\begin{introtheorem}[ \RamiD{See \S\ref{sec:cusp}}]\label{thm:intro.cusp}
Let $\rho$ be \RamiG{an irreducible cuspidal} representation of $G$. Let $\fX_G$ be the torus of unramified characters of $G$. Then there is a smooth subvariety $\fX' \subset \fX_G$  \RamiG{such that} $$p_X(\chi\rho)=\begin{cases}0 & \text{if }\chi\notin\fX' \\
(1+t)^{\dim \fX_G-\dim \fX'}& \text{if }\chi\in\fX' \\
\end{cases}$$
\end{introtheorem}
In order to prove  Theorems \ref{thm:intro.ep} and \ref{thm:intro.cusp} above we introduce, for any cuspidal data (i.e. a Levi subgroup $\bfM<\bfG$ and \RamiG{an irreducible cuspidal} representation $\rho$ of $M=\bfM(F)$%
)  a \EitanG{coherent} sheaf $\cF_{X}((M,\rho))$ over the torus $\fX_M$ of unramified characters of $M$. We prove:
\begin{introtheorem}[See Theorem \ref{lem: mutl_via_fiber} and Theorem \ref{thm: StrCuspShf} \Rami{for more general results}]\label{thm:intro.mutl_via_fiber}
$ $
\begin{enumerate}
\item\label{thm:intro.mutl_via_fiber:1} In the notations of Theorem \ref{thm:intro.ep} we have
$m^i_{X}(\pi_\chi)=\dim\Ext^i_{O_{\fX_M}}(\cF_{X}((M,\rho)),\delta_\chi),$ where $\delta_\chi$ is the skyscraper sheaf at $\chi\in \fX_M$
\Eitan{and $O_{\fX_M}$ is the sheaf of regular functions on the torus of unramified characters $\fX_M$}
\item\label{thm:intro.mutl_via_fiber:2} In the notations of Theorem \ref{thm:intro.cusp} we have $\cF_{X}((G,\rho))=i_*(\cL)$, where \RamiG{$X'$ is as above and} $i:\fX'\to\fX_G$ is \Eitan{the embedding} and $\cL$ is a locally free sheaf over $\fX'$
\end{enumerate}
\end{introtheorem}
\RamiH{
Infact, for every $V\in \cM(G)$ we introduce a \EitanG{quasi-coherent} sheaf $\cF_{V}((M,\rho))$ which coincide with $\cF_{X}((M,\rho))$ when $V=\Sc(X)$. We prove analogous results for this sheaf.

In Appendix \ref{sec:app} we prove a generalization of theorem \ref{thm:intro.ep} where we replace the cuspidal representation $\rho$ with a general admissible representation $\tau$. For this we need to generalize the construction of the sheaf $\cF_{X}$ on $\fX(M)$ to the case when the representation of $M$ is not cuspidal. We are thus led to introduce an object of the derived category of sheaves on $\fX(M)$, denoted $\cG_{V}(M,\tau)$, that plays a similar role to the that of $\cF_{V}((M,\rho))$.

\begin{introtheorem}[See Appendix \ref{sec:app} for a more general result]\label{thm:intro.der_mutl_via_fiber}
$ $
Let $\bfM\subset \bfG$ be a Levi subgroup of $\bfG$. Let $\tau$ be {an admissible} representation of $M=\bfM(F),$ and let $\chi$ be an unramified character  of $M$. Let $V\in\cM(G)$. \EitanG{Then the following holds:}
\begin{enumerate}
\item\label{thm:intro.der_mutl_via_fiber.1} $\RHom_{G}(V,i_M^G(\chi \tau))\cong \RHom_{O(\RamiD{\CH_M})}(\cG_{V}(M,\tau), \delta_{\chi}).$
\item\label{thm:intro.der_mutl_via_fiber.2} If $\tau$ is cuspidal then $\cG_{V}(M,\tau)\cong\cF_{V}((M,\tau))$.
\item\label{thm:intro.der_mutl_via_fiber.3}  If $V=\Sc(X)$ \EitanG{then} $\cG_{V}(M,\tau)$ is perfect.
\item\label{thm:intro.der_mutl_via_fiber.4}  $e_{X}(i_M^G(\chi \tau))$ is independent of $\chi$.
\end{enumerate}
\end{introtheorem}
Our approach to Theorem \EitanG{\ref{thm:intro.der_mutl_via_fiber}} is based on a slight modification of the the notion of \RamiK{multiplicities}. Namely we replace $\Hom$ by the tensor product. The relation between the two notions of  multiplicity is given in Corollary \ref{cor:tens.mult}
}

\subsection{Structure of the paper} In \S\ref{sec:prel} we give  a brief overview  on homological multiplicities and recall few properties of the Bernstein decomposition and Bernstein center.
In \S\ref{sec:hom.fin}  we prove Theorem \ref{thm:intro.f.mul}. In \S\ref{sec:mult.shif} we prove Theorem \ref{thm:intro.mutl_via_fiber}\eqref{thm:intro.mutl_via_fiber:1}. In \S\ref{sec:eu} we prove Theorem \ref{thm:intro.ep}. In \S\ref{sec:cusp}  we prove theorem \ref{thm:intro.mutl_via_fiber}\eqref{thm:intro.mutl_via_fiber:2}
\RamiD{and deduce} Theorem \ref{thm:intro.cusp}. \RamiH{
In Appendix \S\ref{sec:ten.mult} we study the relation between tensor multiplicity and usual multiplicity.
In Appendix \S\ref{sec:app} we introduce the object $\cG_{V}(M,\tau)$, and prove Theorem \ref{thm:intro.der_mutl_via_fiber}.}
\subsection{Acknowledgments}
We thank Dipendra Prasad for a number of inspiring lectures on Ext braching laws. We also thank
Dmitry Gourevitch for many useful discussions \RamiE{and Yotam Hendel for his careful proof reading}. 

This project was conceived while both authors were in Bonn as part of the program ``Multiplicity problems in harmonic analysis" in the Hausdorff Research Institute for Mathematics.
\RamiE{A.A. was partially supported by ISF grant 687/13,
and a Minerva foundation grant.}

\section{Preliminaries}\label{sec:prel}
\subsection{Homological multiplicities}\label{ssec:prel.hom}
Fix $H \subset G$ a closed subgroup.
\Eitan{As in the case of multiplicity spaces, one can define  homological multiplicity  in several ways.}  The relation between  them  is given by the following:
\begin{proposition}[{\cite{Cas},\cite[Proposition 2.5]{Pra}}]
Let $\pi$ be an admissible representation of $G$ and $\eta$ a character of $H.$ Then
$$\RamiK{\Ext}^i_H(\pi|_H,\eta)=\RamiK{\Ext}^i_{G}(\pi,Ind_{H}^{G}(\eta))=\RamiK{\Ext}^i_G(ind_{H}^{G}(\eta'),\tilde{\pi}),$$ where $\eta'=\eta^{-1} \Delta_{H}^{-1}$ and $\Delta_{H}$ is the modulus character of $H.$
\end{proposition}

Following this proposition we introduce some notation:

\begin{definition}[Homological Multiplicities]
In the setting of the proposition above we define
$$\widetilde m_{(H,\eta)}^{i}(\pi):=\dim \RamiK{\Ext}_{H}^{i}(\pi,\eta) \text{ and }  m_{(H,\eta)}^{i}(\pi):=\dim \RamiK{\Ext}_{G}^{i}(ind_{H}^{G}(\eta),\pi).$$
If $\eta$ is trivial, we will omit it from the notation.
\end{definition}
To relate those notions to the ones introduced in \cite{Pra} we need:
\begin{proposition}
Let $\pi_1$ be a smooth representation of $G$ and $\pi_2$ be an  admissible representation of $H$. Then
$\Ext_H^i(\pi_1,\pi_2)\cong\Ext_H^i(\pi_1\otimes\tilde\pi_2,\C).$ In particular we have
$$\dim\Ext_H^i(\pi_1,\pi_2)=\tilde m_{H}^{i}(\pi_1 \otimes  \tilde \pi_2).$$
\end{proposition}
\begin{proof}
$ $
\begin{enumerate}[{Step} 1]
\item Proof for $i=0$.\\
We have
$\Hom_\C(\pi_1,\tilde\pi_2^*)\cong Bil(\pi_1,\tilde\pi_2;\C) \cong\Hom_\C(\pi_1\otimes\tilde\pi_{2},\C)$.
We get
$$\Hom_H(\pi_1,\pi_2)\cong\Hom_H(\pi_1,\tilde{\tilde\pi}_2)=\Hom_H(\pi_1,\tilde\pi_2^*)\cong\Hom_H(\pi_1\otimes\tilde\pi_{2},\C).$$
\item Proof for the general case.\\
To deduce the general case from the previous step we need to show that there are enough projective representations of $G$ that stay projective after restrictions to $H$. This follows from the fact that  $\Sc(G)\otimes V$ is a projective representation of $G$ for any vector space $V$ (See \cite{Cas}, \cite[Propositions 2.5,2.6]{Pra} for more details).
\end{enumerate}
\end{proof}
\subsection{Bernstein center}
We summarize the parts of the theory of Bernstein center and Bernstein decomposition that are used in this paper.
\EitanB{We need few notations.
\begin{notation}
Let $\bf P$ be a \RamiG{parabolic subgroup} of $\bf G$ and  $\bfM$ be its Levi factor. Let $\rho$ be a cuspidal representation of $M=\bfM(F)$. Denote
\begin{itemize}
\item $\RamiE{i^G_M}: \cM(M) \to \cM(G)$ the (normalized) parabolic induction of $M$ w.r.t. $P$
\item $ \RamiE{\overline i^G_M}(\rho): \cM(M) \to \cM(G)$   the (normalized) parabolic induction from $M$ w.r.t. an opposite \RamiG{parabolic subgroup}  $\bar \bfP$.
\item
 $\RamiE{r^G_M}:  \cM(G) \to \cM(M)$  the (normalized) Jacquet functor.
\end{itemize}
\end{notation}
}
\EitanG{
Given $M\subset G$, the objects defined above depend on the choice of a parabolic $P \subset G.$ However, unless stated otherwise the result will not depend on this choice and we will ignore it. 
\begin{theorem} [Bernstein]\label{thm:cent}
$ $
\begin{enumerate}
\item\label{Ber.cent:basis}
There exists a local base $\mathcal{B}$ of the topology at the unit $e \in G,$ such that every $K \in \cB$ is a compact (open) subgroup satisfying the following:

\begin{enumerate}

\item \label{Ber.cent:1}

The category $\cM(G,K)$ of representations generated by their $K$-fixed vectors is a direct summand of $\cM(G).$

\item \label{Ber.cent:2}

The functor $V \to V^{K}$ is an equivalence of categories from $\cM(G,K)$ to the category of modules over the Hecke algebra  $\cH(G,K)$.
\item\label{Ber.cent:not} The algebra  $\cH(G,K)$ is Noetherian.

\item For any Levi subgroup $M \subset G$ and for any $V \in \cM(G)$ the map
$$V^K\to r_{M}^G(V)^{K\cap M}$$ is onto. 
\end{enumerate}
We will call an open compact subgroup $K$ satisfing these properties, a splitting subgroup.

\item\label{Ber.cent:2.5} \RamiE{The f}unctor $\RamiE{r^G_M}$ is right adjoint to $\RamiE{\overline i^G_M}$, that is $$\RamiK{\Hom_G}(\RamiE{\overline i^G_M}(V),W) \cong \RamiK{\Hom_M}(V,r^G_M(W)).$$

\item\label{Ber.cent:3}  For a cuspidal data $(M,\rho)$ let \EitanB{$\Psi_{G}(M,\rho)=\RamiE{\overline i^G_M}(\rho \otimes O(\fX_M))$} be the normalized parabolic induction of $\rho \otimes O(\fX_M)$ where the action of $M$ is diagonal. Then $\Psi_{G}(M,\rho)\in \cM(G)$ is a projective generator of a direct summand of the category $\cM(G).$
\item \label{thm:ber:azu0}  Let $\rho$ be \RamiG{an irreducible cuspidal} representation of $G$, and set \Eitan{$ \fI_\rho:=\{\psi\in\fX_G|\psi \rho\simeq \rho\}$}  and embed $\cO(\fX_G)$ in $\cR_{(G,\rho)}:=\End(\Psi_{G}(G,\rho))$.
Then there exist\RamiE{s} a decomposition $$\cR_{(G,\rho)}= \bigoplus_{\psi\in \fI_\rho}O(\fX_G)\nu_\psi,$$ \RamiG{such that}  $\nu_\psi f=f_\psi \nu_\psi $ and $\nu_\psi\nu_{\psi'}=c_{\psi,\psi'}\nu_{\psi\psi'}$ where $f_\psi$ is  the translation of $f\in O(\fX_G)$ by $\psi$ and  $c_{\psi,\psi'}$ are scalars.

\end{enumerate}

\end{theorem}
}
For completeness we supply exact references:
\begin{itemize}
\item For statement(\ref{Ber.cent:basis}) See  \cite[Corollary 3.9, Proposition 2.10 and its proof, Corollary 3.4, the proof of Theorem 2.13 and Proposition 3.5.2]{BD}.


\RamiD{\item For  statement (\ref{Ber.cent:2.5}) see \cite{BerSA}.}
\item Statement (\ref{Ber.cent:3}) follows from \cite[Propositions 34,35]
{BerLec} and \cite[Proposition 2.10]{BD}.
\item For  statement  (\ref{thm:ber:azu0}) see \cite[Proposition 28]{BerLec}.
\end{itemize}


%
%
%

\section{Finiteness of Homological multiplicities}\label{sec:hom.fin}

In this section we prove a \RamiD{generalization} of Theorem \ref{thm:intro.f.mul}. Namely,

\begin{theorem}\label{thm.f.mul}
 Assume that $G/H$ is an $F$-spherical $G$-variety. Let $\eta$ be a character of \RamiE{$H.$} Assume that for any irreducible  representation $(\pi,V)$ we have ${m}_{(\RamiE{H},\eta)}^0(\pi)< \infty.$ Then for any admissible representation $\pi$ and any natural $j$ we have $ m_{(\RamiE{H},\eta)}^j(\pi)< \infty.$
\end{theorem}
We will say the the pair $(G,H)$ and the $G$-space $G/H$ is of \textbf{finite type}, if the conditions of Theorem \ref{thm.f.mul} are satisfied for any $\eta$. As mentioned earlier, it is conjectured that any $F$-spherical pair is of finite type and it is \RamiG{known} in many cases.

%

For the proof we recall the following facts:

\begin{theorem} [{\cite{AAG}, \cite[Appendix B]{AGS_Z}}]\label{thm.f.gen}
Assume \RamiG{that} $G/H$ is an $F$-spherical $G$-variety and that    for any irreducible  representation $(\pi,V)$ we have ${m}_{(H,\eta)}^0(\pi)< \infty.$ Let $K$ be an open subgroup of $G.$  Then the module $ind_{H}^{G}(\eta)^{K}$ is finitely generated over $\cH(G,K).$
\end{theorem}

\begin{proof}[Proof of Theorem \ref{thm.f.mul}]

Let $K$ be such \RamiE{that} $V^{K}$ generates $V.$ \EitanH{By Theorem 
\ref{thm:cent}(\ref{Ber.cent:basis})}
\EitanH{we can assume that $K$ is a splitting subgroup.}

Let $W$ be the sub-representation of \EitanB{$ind_{H}^{G}(\eta)$} generated by $K$ fixed vectors.
We have

$$ m_{(H,\eta)}^{\RamiD{i}}(\pi)=\dim \Ext_{G}^{i}(ind_{H}^{G}(\eta),\pi)=\dim \Ext_{G}^{i}(W,\pi)=\dim \Ext_{\cH(G,K)}^{i}(ind_{H}^{G}(\eta)^K,\pi^K).$$

The fact that $\cH(G,K)$ is Noetherian and Theorem \ref{thm.f.gen} \RamiE{imply} that \RamiD{the module $ind_{H}^{G}(\eta)^K$ has a resolution by finitely generated free $\cH(G,K)$-module. Thus} $$\dim \Ext_{\cH(G,K)}^{i}(ind_{H}^{G}(\eta)^K,\pi^K)<\infty.$$
\end{proof}

\section{Homological multiplicities and the multiplicity sheaf}\label{sec:mult.shif}
In this section we study the homological multiplicities of cuspidally induced representations.

\begin{defn}\label{not: sheaf}
Let \EitanG{$M,\rho$} as \RamiE{Theorem \ref{thm:cent}\eqref{Ber.cent:3}}.
 Let $V \in \cM(G)$. We define
\EitanB{
$$\cF_{V}((M,\rho)):=\Hom_{G}(\Psi_{G}(M,\rho),V),$$

This is a module over $\End_{\RamiE{G}}(\Psi_{G}(M,\rho))$ and thus \RamiE{a module} over $\cR_{(M,\rho)}=\End_{\RamiE{M}}(\Psi_{M}(M,\rho))$. In particular, it is a module over $\cO(\CH_M)$.
We shall view it as a quasi-coherent
sheaf over the space  $\CH_M$ of unramified characters of $M$.
}

\end{defn}

\begin{notation}

Let $P,M,\rho$ as above.\begin{itemize}
\item When  $V=\RamiD{ind_{H}^{G}}(\eta)$ for a character $\eta$ of a subgroup $H\subset G$ we denote  $\cF_{H,\eta}((M,\rho)):=\cF_{V}((M,\rho))$.
\item  When $V=\Sc(X,\cL)$ for a $G$-equivariant sheaf $\cL$  on $X$ we will denote  $\cF_{X,\cL}((M,\rho)):=\cF_{V}((M,\rho))$.
\item If $\eta$ or $\cL$ are trivial we will omit them from the notation.

\end{itemize}
\end{notation}



The following generalizes Theorem \ref{thm:intro.mutl_via_fiber}\eqref{thm:intro.mutl_via_fiber:1} of the introduction.
\begin{theorem}\label{lem: mutl_via_fiber}
\EitanG{Let $V \in \cM(G)$}
Then,

$$\RamiK{\Ext}_{G}^{i}(V,i_M^G(\chi \rho))\cong \RamiK{\Ext}^i_{O(\RamiD{\CH_M})}(\cF_{V}((M,\rho)), \delta_{\chi}),$$

 where $\delta_{\chi}$ is the skyscraper sheaf over $\chi \in \RamiD{\CH_M}.$
\end{theorem}

For the proof we need the following standard lemma:

\begin{lemma}\label{lem:Azum}
Let $Z$ be a commutative finitely generated algebra over $\C$ without nilpotent elements. Let $A$ be an algebra over $Z$ and let $M$ be \EitanG{an} $A$-module.
Let $\spec(C)\to\spec(Z)$ be an \et\  map.
Let $\Delta$ be an irreducible $A$-module with annihilator $\fm$. Let $\fm'$ be a maximal ideal of $C$ lying over $\fm$. Let $\Delta':=(\Delta\otimes_Z C)/\fm'$ \EitanG{considered as a $A \otimes_Z C$-module}.
Then $$ \Ext^*_{A}(M,\Delta)\cong \Ext_{A \otimes_Z C}^*(M\otimes_Z C,\Delta').$$

\end{lemma}
\begin{proof}
Since $C$ is flat over $Z$ we have isomorphism of $C$ modules $$\Ext_{A \otimes_Z C}^*(M\otimes_Z C,\Delta \otimes_Z C)\cong \Ext^*_{A}(M,\Delta)\otimes_Z C.$$
\EitanG{Since $\fm$ annihilates $\Delta$ and $\Ext^*_{A}(M,\Delta)$, we have $$\Delta \otimes_Z C \cong \Delta \otimes_{Z/\fm} C/\fm \text{ and } \Ext^*_{A}(M,\Delta) \otimes_Z C \cong \Ext^*_{A}(M,\Delta)\otimes_{Z/\fm} C/\fm.$$
Thus,} we obtain $$\Ext_{A \otimes_Z C}^*(M\otimes_Z C,\Delta \otimes_{Z/\fm} C/\fm)\cong \Ext^*_{A}(M,\Delta)\otimes_{Z/\fm} C/\fm.$$
Since \RamiE{$C/\fm'$} is a direct summand of $C/\fm$ we obtain
$$\Ext_{A \otimes_Z C}^*(M\otimes_Z C,\Delta \otimes_{Z/\fm} C/\fm')\cong \Ext^*_{A}(M,\Delta)\otimes_{Z/\fm} C/\fm'.$$
This implies the assertion.
\end{proof}

We will also need the following proposition:
\begin{proposition}\label{prop:et}
Let $\rho$ be a \RamiG{an irreducible cuspidal} representation of $G$, recall that $$\fI_\rho:=\{\chi\in\fX_G|\chi \rho\simeq \rho\}.$$
 Then there exists an onto \et\ map $\spec(C)\to\spec(Z(\cR_{(G,\rho)}))$ such that
 the triple $$C\subset C\otimes_{Z(\cR_{(G,\rho)})}O(\fX_G)\subset  C\otimes_{Z(\cR_{(G,\rho)})}\cR_{(G,\rho)}$$ is isomorphic to $$C \subset\left\{\left\zerodel  \begin{pmatrix}f_{1} & \ &  \\
 & \ddots &  \\
 & & f_{n} \\
\end{pmatrix}\right| f_i \in C \right\}\subset Mat_n(C).$$
\end{proposition}

\begin{proof}
By Theorem \ref{thm:cent}(\ref{thm:ber:azu0}) we have $$\cR_{(G,\rho)}= \bigoplus_{\psi\in \fI_\rho}O(\fX_G)\nu_\psi,$$
Thus the center of $\cR_{(G,\rho)}$ is $\cO(\fX_G/\fI_\rho) \subset \cO(\fX_G).$ 
Let $C=O(\fX_G)$. Let $$\phi:\ O(\fX_{G}) \otimes_{O(\fX_{G}/\RamiE{\fI}_{\rho})} O(\fX_{G}) \to O(\fX_{G} \times \RamiE{\fI}_{\rho})$$ be the map given by $\phi(f_{1} \otimes f_{2})(x,j)=f_{1}(x)f_{2}(jx).$ It is easy to see that $\phi$\ is an isomorphism. This yields an identification  $$C\otimes_{Z(\cR_{(G,\rho)})}\cR_{(G,\rho)} \cong\bigoplus_{\psi\in \fI_\rho} O(\fX_{G} \times \RamiE{\fI}_{\rho})  \nu_\psi.$$
Define a map $\mu: \bigoplus_{\psi\in \fI_\rho} O(\fX_{G} \times \RamiE{\fI}_{\rho})  \nu_\psi\to \End_{O(\fX_{G} )}(O(\fX_{G} \times \RamiE{\fI}_{\rho}))$  by
\begin{equation}
\mu(\nu_\psi)(f)(\chi,\psi')=c_{\psi,\psi'}f(\chi,\psi^{{-1}}\psi')
\end{equation}
 and for $g\in  O(\fX_{G} \times \RamiE{\fI}_{\rho})$
\begin{equation}\label{eq:cart}
\mu(g)(f)(\chi,\psi)=g(\chi,\psi)f(\chi,\psi).
\end{equation}
It remains to show that $\mu$ is an isomorphism. Fixing a point $\chi \in \fX_G$,
Burnside's theorem on matrix algebras implies that the specialization $$\mu|_\chi:\left\zerodel\left(C\otimes_{Z(\cR_{(G,\rho)})}\cR_{(G,\rho)}\right)\right|_\chi \to \End(\C[\fI_\rho]) $$ is onto, and hence an isomorphism.  This implies that $\mu$ is isomorphism.

\end{proof}

\begin{proof}[{Proof of Theorem \ref{lem: mutl_via_fiber}}]
\RamiD{We will compute the left and the right hand side of the desired isomorphism and see that \RamiE{they are} the same. We start with the left hand side.

}
Deriving the Frobenius reciprocity we have
\RamiD{$$\Ext_{G}^{i}(V,i_M^G(\chi \rho))=  \Ext_{M}^{i}(r_M^G(V),\chi \rho).$$}
By 
{Theorem} \ref{thm:cent} \eqref{Ber.cent:3}

\EitanB{
$$
\Ext_{M}^{i}(r_M^G(V),\chi \rho)= \Ext^{i}_{\cR_{(\RamiG{M},\rho)}}(\Hom_M(\Psi_{M}(M,\rho), r_M^G(V)),\Hom_M(\Psi_{M}(M,\rho), \chi \rho)),
$$
}

\EitanB{Note that by the second adjointness theorem (Theorem \ref{thm:cent} \eqref{Ber.cent:2.5}) we have $$\Hom_M(\Psi_{M}(M,\rho), r_M^G(V))=\cF_{V}((M,\rho)).$$}

Furthermore, \EitanB{$\Delta:=\Hom_M(\Psi_{M}(M,\rho), \chi \rho)$} is an irreducible module over ${\cR_{(\RamiG{M},\rho)}}$.
Let $C$ be as in Proposition \ref{prop:et} \RamiG{(applied for the group $M$)}. Let $\Delta'$ be a ${\cR_{(\RamiG{M},\rho)}}':={\cR_{(\RamiG{M},\rho)}}\otimes_{Z(\cR_{(\RamiG{M},\rho)})}C$-module which coincides with $\Delta$ as  a $Z(\cR_{(\RamiG{M},\rho)})$-module. By Lemma \ref{lem:Azum} we have

{
$$\Ext^{i}_{\cR_{(\RamiG{M},\rho)}}(\cF_{V}((M,\rho)),\Delta)\cong \Ext^{i}_{{\cR_{(\RamiG{M},\rho)}}'}(\cF_{V}((M,\rho))\otimes_{Z(\cR_{(G,\rho)})}C,\Delta').$$
}

Identify ${\cR_{(\RamiG{M},\rho)}}'\cong Mat_n(C)$ as in Proposition \ref{prop:et}.  We get that there is a $C$-module \EitanB{$T$  \RamiG{such that} $\cF_{V}((M,\rho))\otimes_{Z(\cR_{(\RamiG{M},\rho)})}C=T^{n}$.} This implies that

\EitanB{
$$\Ext^{i}_{{\cR_{(\RamiG{M},\rho)}}'}(\cF_{V}((M,\rho))\otimes_{Z(\cR_{(\RamiG{M},\rho)})}C,\Delta')\cong \Ext^{i}_{C}(\EitanB{T},\delta_{\chi'}),$$
}
where $\chi'$ is the character with which $C$ acts on $\Delta'$.

\RamiD{To compute the right hand side we}
\EitanB{apply Lemma \ref{lem:Azum}  again and obtain

$$\RamiK{\Ext}^{i}_{O(\CH_M)}(\cF_{V}((M,\rho)),\delta_\chi)\cong
\RamiK{\Ext}^{i}_{O(\CH_M) \otimes_{Z(\cR_{(\RamiG{M},\rho)})}C}(\cF_{V}((M,\rho))\otimes_{Z(\cR_{(\RamiG{M},\rho)})}C,\delta_{\chi''}),$$
}
\RamiD{where $\chi''$ is a character of $O(\CH_M) \otimes_{Z(\cR_{(\RamiG{M},\rho)})}C$ whose restriction to $C$ is $\chi'$.
Since $\cF_{V}((M,\rho))\otimes_{Z(\cR_{(G,\rho)})}C=T^{n}$ we get
$$\RamiK{\Ext}^{i}_{O(\CH_M) \otimes_{Z(\cR_{(\RamiG{M},\rho)})}C}(\cF_{V}((M,\rho))\otimes_{Z(\cR_{(\RamiG{M},\rho)})}C,\delta_{\chi''})\cong \RamiK{\Ext}^{i}_{C}(\EitanB{T},\delta_{\chi'}).$$}

\end{proof}

\section{The Euler Characteristics}\label{sec:eu}
Recall that the homological dimension of  the abelian category $\cM(G)$ is finite (See \cite[Theorem 29]{BerLec}).  This allows us to give the following \EitanG{notation:
\begin{notn}[cf. \cite{Pra}] $ $
Let $\pi, \tau \in \cM(G)$. Define $$EP_G(\pi,\tau)=\sum (-1)^i \dim \Ext^i_{\RamiE{G}}(\pi, \tau).$$
\end{notn}
In order to formulate the main result of this section we need:

\begin{defn}
We will call $V \in \cM(G)$ locally finitely generated if
for any compact open $K<G$ the module  $V^K$ is finitely generated over $\cH(G,K)$.  
\end{defn}
}

\RamiD{The following is a generalization of Theorem \ref{thm:intro.ep}.}
\begin{theorem}\label{thm:ep}
Let $\bfP$ be a \RamiG{parabolic subgroup} of $\bfG$ and  $\bfM$ be its Levi factor. Let $\rho$ be \RamiG{an irreducible cuspidal} representation of $M=\bfM(F)$ and  $\chi \in \RamiD{\CH_M}$ be an unramified character of $M$. \EitanG{Let $V \in \cM(G)$ locally finitely generated  representation.

Then
$EP_G(V,ind_M^G(\chi \cdot \rho))$ is constant as a function of $\chi$.}
\end{theorem}

This theorem follows from Theorem \ref{lem: mutl_via_fiber} and the following 
\EitanG{propositions:}
\begin{proposition}
Let $\cF$ be a coherent \RamiD{sheaf} over \RamiD{a} (complex) smooth algebraic variety $X$. Then the function $x\mapsto \sum (-1)^i \dim \RamiK{\Ext^i_{O_X}}(\cF, \delta_x)$ is \RamiD{a} locally constant function on $X$.
\end{proposition}
\EitanG{
\begin{proof}
Since $X$ have finite homological dimension, it is enough to prove the proposition for localy free $\cF$. In this case the proposition is trivial.
\end{proof}
\begin{proposition}
Let $V \in \cM(G)$ be a locally finitely generated  representation and let $(M,\rho)$ be a cuspidal data. Then $\cF_V((M,\rho))$ is coherent.
\end{proposition}
\begin{proof}

Denote the category generated by $\Psi_G(M,\rho)$ by $\cM_{(M,\rho)}(G)$. 
Since $\Psi_G(M,\rho)$ is finitely generated there exist \EitanH{an} open compact subgroup $K<G$ such that $\cM_{(M,\rho)}(G) \subset \cM_{K}(G).$ \EitanH{By Theorem \ref{thm:cent}(\ref{Ber.cent:basis}) we can assume that $K$ is splitting.}

Let $V_{(M,\rho)}$ be the maximal subrepresentation of $V$ which is contained in $\cM_{(M,\rho)}(G)$ and
let $\Sc(G/K)_{(M,\rho)}$ be the maximal subrepresentation of $\Sc(G/K)$  which is contained in  $\cM_{(M,\rho)}(G)$. By \EitanH{Theorem \ref{thm:cent}(\ref{Ber.cent:1}, \ref{Ber.cent:3}),} the object  $\Sc(G/K)_{(M,\rho)}$ is a projective genarator of $\cM_{(M,\rho)}(G)$. Since $V$ is
locally finitely generated,  $V_{(M,\rho)}^K=\RamiK{\Hom_G}(\Sc(G/K),V_{(M,\rho)})$ is finitely generated over
$\cH_K(G)=\End(\Sc(G/K))$. Thus it is finitely generated over $\End(\Sc(G/K)_{(M,\rho)}).$
Thus, by Theorem \ref{thm:cent}(\ref{Ber.cent:not}) it is finitely presented over $\End(\Sc(G/K)_{(M,\rho)})$. Therefore by \cite[Satz 3]{Len},   
$\RamiK{\Hom_G}(\Psi_G(M,\rho),V_{(M,\rho)})=\cF_V((M,\rho))$ it is finitely presented over $\End(\Psi_G(M,\rho))$ and thus over $\fX(M)$.
\end{proof}
}
%
%
%
%
%
%
\section{cuspidal case}\label{sec:cusp}
The following theorem \RamiD{g}eneralizes theorem \ref{thm:intro.mutl_via_fiber}\eqref{thm:intro.mutl_via_fiber:2}\RamiD{.}
\begin{theorem}\label{thm: StrCuspShf}
Let $\rho$ be \RamiG{an irreducible cuspidal} representation of $G$. Let $\bfH<\bfG$ be a Zariski closed subgroup and $H=\bfH(F)$. Let  $X=G/H$  and let $\cL$ be a $G$-equivariant line bundle over $X$. Assume that $X$  is of finite type.

 Then, {t}he sheaf
$\cF_{X,\cL}(G,\rho)$ is a direct image of a locally free
sheaf on a smooth subvariety of $\RamiD{\CH_G}$.
\end{theorem}

For the proof we will need the following lemma{s}
\begin{lemma}\label{lem:FinEt}
Let $\phi:S \to S'$ be finite \et\ map of algebraic varieties.
Let $\cF$ be a coherent sheaf over $S$. Suppose that $\phi_*(\cF)$
is  a direct image of a locally free sheaf on a smooth subvariety
of $S'$. \RamiK{Then,} $\cF$ is  a direct image of a locally free sheaf
on a smooth subvariety of $S$.
\end{lemma}
\begin{proof}
Without loss of generality we can assume that $\phi_*(\cF)$ is a
locally free sheaf on $S'$. Now recall that a sheaf is locally
free if and only if it is locally free in the \et\ topology.
So we can assume that the map $\phi$ is a projection from a
product of $S'$ by a {reduced zero dimensional} scheme. In this case the assertion
follows from the fact that a direct summand of a locally free
sheaf is locally free.
\end{proof}

{
\begin{lemma}\label{lem:FinInd}
Let $\Lambda_\cZ=\cX_{*}(\cZ(\bfG))$ be the lattice of co-characters of the center of $G$.
{C}onsider the map $\Lambda_{\cZ} \to G$ given by evaluation at  the uniformizer $\RamiG{\varpi},$ and consider   $\Lambda_{\cZ}$ as a subset of $G$. Let $X,H$ be as in Theorem \ref{thm: StrCuspShf}.
Then $\Lambda_{\cZ} \cap H$ has finite index in $\Lambda_{\cZ} \cap (H \cdot {G^0}),$ where ${G^0}$ is the subgroup generated by compact subgroups of $G$.
\end{lemma}
\begin{proof}
Let  ${H^0}$ \RamiE{be the subgroup generated by compact subgroups of $H$. We also denote  $H^1:=H^0\cZ(H)$ and $G^1=G^0\cZ(G)$.} Let $\Lambda_{\cZ(H)}=\cX_{*}(\cZ(\bf H)).$ We have the following commutative diagram:
$$\xymatrix{
\Lambda_{\cZ} \cap (H\cdot \Rami{G^0})\ar[r]^{ \quad\ \ \ } &\Lambda_{\cZ}\ar[r]^{i_G} & G \ar[r]^{p_G} \ar@{<-}[d]^{i} & G/\Rami{G^0} \ar@{<-}[d]^{i_0}\\
&& H \ar@{<-}[d]^{j} \ar[r]^{p_H} & H/\Rami{H^0} \ar@{<-}[d]^{j_{0}}\\
\Lambda_{\cZ} \cap H\ar[r]^{} \ar[uu]^{}&\Lambda_{\cZ(H)}\ar[r]^{i_H} & H^{1} \ar[r]^{p_{H^{1}}} & H^{1}/H^{0}
} $$
We have $\Lambda_{\cZ} \cap (H\cdot \Rami{G^0})=(p_G\circ i_G)^{-1}(Im (i_0))$. \RamiH{Note that $\cZ(G) \cap {G^0}$ is compact. 
Hence $\Lambda_\cZ \cap \Rami{G^0}$ is trivial. Thus}
we have $\Lambda_{\cZ} \cap H=(p_G\circ i_G)^{-1}(Im (i_0 \circ j_0\circ p_{H^{1}}\circ i_H )).$ {Since $ p_{H^{1}}\circ i_H $ is onto, we have $\Lambda_{\cZ} \cap H=(p_G\circ i_G)^{-1}(Im (i_0 \circ j_0 )).$} The assertion follows now from the fact that $Im(j_0)$ has finite index in $H/\Rami{H^0}.$
\end{proof}
}

\begin{proof}[Proof of Theorem \ref{thm: StrCuspShf}]
$ $
\begin{enumerate}[{Step }1.]
\item Proof for the case when $G=G^1:=G^0\cZ(G)$
\\
{Let \RamiD{$\Lambda_\cZ$} be} as in the above lemma (Lemma \ref{lem:FinInd}).
{D}ecompose $G=G^{0} \times \Lambda_{\cZ}.$
{ Let $\Lambda^{0}_{\cZ}:=\Lambda_{\cZ} \cap H$ and $\Lambda^{1}_{\cZ}:=\Lambda_{\cZ} \cap (H \cdot \Rami{G^0})$. We can \RamiE{find a decomposition} $\Lambda_{\cZ}:=\Lambda^{2}_{\cZ} \oplus \Lambda^{3}_{\cZ},$ \RamiG{such that}  $\Lambda^{1}_{\cZ}$ is a subgroup of finite index in  $\Lambda^{2}_{\cZ}$. We define $X^0:=\Lambda_{\cZ}^{2} \cdot G^{0} \cdot [e],$ where $[e]\in X$ is the class of the unit element in $G$.}

{Using the fact that $G^0$ is normal in $G$, we get}
  $$X \cong X^0 \times \Rami{\Lambda^{3}_{\cZ}}$$
  as $\Rami{G^{0} \cdot \Lambda^{3}_{\cZ} \cdot \Lambda^{0}_{\cZ} \RamiE{\cong} G^{0} \times \Lambda^{3}_{\cZ} \times\Lambda^{0}_{\cZ}}$-spaces.
{Here the action of $\Lambda^{0}_{\cZ}$ on $X^0 \times \Rami{\Lambda^{3}_{\cZ}}$ is trivial, the action of  $\Lambda^{3}_{\cZ}$ is on the second component and the action of  $G^{0}$ is on the first.

Now {consider} the fiber $\cL|_{\RamiD{[e]}}$ as a character of $G^{0} _{\RamiD{[e]}}\times \Lambda^{0}_{\cZ}$ and decompose it into a product  $\chi_1 \otimes  \chi_2$.}


Thus {{w}e have isomorphisms of $G^{0} \times \Lambda^{3}_{\cZ} \times\Lambda^{0}_{\cZ}$ representations: $$\RamiD{\cF_{X,\cL}((G,\rho))}= \Hom_{G^{0}}(\rho,\Sc(X\Rami{,\cL}))\cong \Hom_{G^{0}}(\rho,\Sc(X^0,\cL|_{X^0}))
 \otimes {\C}[\Lambda^{3}_{\cZ}]
\otimes \chi_2. $$

Let $L:=\spec({\C}[\Lambda^{3}_{\cZ} \times\Lambda^{0}_{\cZ} ])$. By Lemma \ref{lem:FinInd} the map  $$\pi:\RamiD{\CH_G}=\spec({\C}[G/G^0])\to L$$ is a finite \et\ morphism.  {W}e see that $\pi_*(\RamiD{\cF_{X,\cL}((G,\rho))})$ is  a direct image of a  free sheaf on $\spec( {\C}[\Lambda^{3}_{\cZ}]) \RamiD{\times \{\chi_2\}}$ which is
  a smooth subvariety
of $L$. Thus Lemma \ref{lem:FinEt} implies the assertion}{.}



\item Proof for the general case.\\
As in the previous step consider $\Lambda_\cZ \subset G/G^0.$
This gives a finite {\et\ map} 
 $$\CH(G)=\spec({\C}[G/G^0])
\overset{p}\to \spec({\C}[\Lambda_\cZ]) \Rami{\cong {\C}[{G^1}/G^0]}.$$ Similarly to the previous step, the sheaf $p_*(\RamiD{\cF_{X,\cL}((G,\rho))})$ is a direct
image of a \RamiD{locally} free sheaf on a smooth subvariety of
$\spec({\C}[\Lambda_\cZ ])$.
{Again Lemma  \ref{lem:FinEt} implies the assertion}{.}
\end{enumerate}
\end{proof}
\RamiD{
Theorem \ref{thm: StrCuspShf} together with Theorem \ref{lem: mutl_via_fiber},  \RamiD{imply} Theorem \ref{thm:intro.cusp} using the following standard fact:
\begin{lemma}
Let $X$ be a smooth variety and $Y$ be its closed subvariety. \RamiG{C}onsider $O_Y$ as a coherent sheaf over $X$. Then for any $y\in Y$ we have
$$\dim \RamiK{\Ext^i_{O_X}}(O_Y,\delta_y)=\begin{pmatrix}\dim X-\dim Y \ \\
i \\
\end{pmatrix}.
$$

\end{lemma}
}

\appendix

\section{Tensor Multipicities}\label{sec:ten.mult}
Let $V,W \in \cM(G)$ be smooth representations of $G$. Consider $V,W$ as left $\cH(G)$-modules. \EitanG{Since} $\cH(G)$ is equipped with an \EitanG{anti-involution} induced by $g \mapsto g^{-1}$, we can consider $V$ as a right $\cH(G)$-module. Thus we can define $V\otimes_{\cH(G)}W$.
\begin{lemma}
Let $V,W \in \cM(G)$ be smooth representations of $G$. Then we have:
$$(V\otimes_{\cH(G)}W)^*\cong \RamiK{\Hom}_G(V,\tilde W).$$
\end{lemma}
\begin{proof}
Define $\phi:(V\otimes W)^* \to \RamiK{\Hom}_{\C}(V, W^*)$ by
$$\phi(\ell)(v)(w)=\ell(v\otimes w).$$

Since $V\otimes_{\cH(G)} W$ is a quotient of $V\otimes W$  we can consider $(V\otimes_{\cH(G)} W)^*$ as a subset of  $(V\otimes W)^*$
spesificaly $(V\otimes_{\cH(G)} W)^* \cong \{\ell\in (V\otimes_{\cH(G)} W)^*|\ell(g^{-1}v\otimes w)= \ell(v\otimes gw) \} $

It is easy to  see that $\phi((V\otimes_{\cH(G)} W)^*) \subset \RamiK{\Hom}_G(V, W^*)$
and that $\RamiK{\Hom}_G(V, W^*)\cong \RamiK{\Hom}_G(V, \tilde W)$. So we can consider $\phi$ as a map $(V\otimes_{\cH(G)}W)^*\to \RamiK{\Hom}_G(V,\tilde W).$

Now define $\psi: \RamiK{\Hom}_{\C}(V, \tilde W) \to (V\otimes W)^*$ by:
$$\psi(T)(v \otimes w)=\langle T(v), w \rangle.$$
Again 
$\psi(\RamiK{\Hom}_{G}(V, \tilde W))\subset (V\otimes_{\cH(G)}W)^*$.

Finally we notice that
 $\psi\circ\phi=\id$ and  $\phi \circ \psi=\id$. 
\end{proof}

As an immidiate consequence we get:
\begin{cor}\label{cor:tens.mult}
$$(V\otimes^L_{\cH(G)}W)^*\cong \RHom_G(V,\tilde W).$$
\end{cor}


\section{Induction of general \RamiK{families}}\label{sec:app}

In this Appendix we introduce the object $\cG_{V}(M,\tau)$ and
 prove Theorem \ref{thm:intro.der_mutl_via_fiber}. 

\begin{defn}\label{not: sheaff}
Let $\mathbf P \subset \mathbf G$ be a parabolic subgroup  and $\mathbf M$ be its Levi factor. Let $\tau$ be an admissible representation of $M=\mathbf M (F)$ 
 Let $V \in \cM(G)$. We define
$$\cG_{V}(M,\tau):=ind_{M_{0}}^{M}\tilde \tau|_{M_0} \otimes^{L}_{{\cH(M)}} r_M^G(V),$$

and consider it as an object in the derived category of the category of $M/M_0$-modules. Equivalently, we \EitanG{will} consider it as an  object in the derived category of the category of quasi-coherent sheaves over $\fX_M$.
\end{defn}

It is enough to prove Theorem \ref{thm:intro.der_mutl_via_fiber}(\ref{thm:intro.der_mutl_via_fiber.1}-\ref{thm:intro.der_mutl_via_fiber.3}) since  Theorem \ref{thm:intro.der_mutl_via_fiber}(\ref{thm:intro.der_mutl_via_fiber.4}) follows from Theorem \ref{thm:intro.der_mutl_via_fiber}(\ref{thm:intro.der_mutl_via_fiber.1},\ref{thm:intro.der_mutl_via_fiber.3}) as in \S \ref{sec:eu}.





To compare $\cG_{V}(M,\tau)$ with $\cF_{V}((M,\rho))$  we need the following standard lemma:

\begin{lemma}\label{lem:ass}
\EitanG{Let $A$ and B be an associative algebras (not necessarily unital).
Let $M$ be a left $A$-module,  $K$ be a left $B$-module, and  $M$ be an $(A,B)$-bi-module}

Then $$\Hom_{A}(N,M) \otimes_{B} K=\Hom_{A}(N,M \otimes_{B} K).$$ 
\end{lemma}
\begin{cor}\label{lem: Homs}
For any $V,W\in \cM(G)$ we have
$$\Hom_{G}(V,W) \cong \Hom_{G}(V,\Sc(G)) \otimes_{\cH(G)} W.$$
\end{cor}

We will also need:
\begin{lemma}\label{lem: cusp. gen.}
Let $\rho$ be an irreducible cuspidal representation of $G$ 
and \EitanG{recall that} $\EitanG{\Psi_G(\rho,G)}=ind_{G^{0}}^{G}(\rho|_{G^{0}}).$ 
Consider $\Hom_{G}(\EitanG{\Psi_G}(\rho,G),\Sc(G))$ as $G$-representation \EitanG{by letting  $G$ act on $S(G).$}
We have an isomorphism of $G$-representations:
 $$\Hom_{G}(\EitanG{\Psi_G}(\rho,G),\Sc(G)) \cong \EitanG{\Psi_G}(\tilde{\rho},G).$$
\end{lemma}
\begin{proof}
$ $
\begin{enumerate}[Step 1]

\item Proof that  $\tilde{\rho} \cong \Hom_{G^{0}}(\rho,\Sc(G^{0}))$ as $G^{0}$ representations.\\
This follow immidately from \cite[\S\S5.3 Theorem 8]{BerLec}
\item Proof of the lemma.\\
$$\Hom_{G}(\EitanG{\Psi_G(\rho,G)},\Sc(G))=\Hom_{G^{0}}(\rho,\Sc(G))=\Hom_{\cH(G^0)}(\rho,\Sc(G^0) \otimes_{\cH(G^0)} \Sc(G))$$
By Lemma \ref{lem:ass}
$$ \Hom_{\cH(G^0)}(\rho,\Sc(G^0) \otimes_{\cH(G^0)} \Sc(G))= 
\Hom_{\cH(G^0)}(\rho,\Sc(G^0))\otimes_{\cH(G^0)} \Sc(G)=\Hom_{G^0}(\rho,\Sc(G^0))\otimes_{\cH(G^0)} \Sc(G)$$
Finaly, by the previous step
$$\Hom_{G^0}(\rho,\Sc(G^0))\otimes_{\cH(G^0)} \Sc(G)=\tilde{\rho} \otimes_{\cH(G^0)} \Sc(G)=\EitanG{\Psi_G}(\rho,G).$$



 






\end{enumerate}
\end{proof}




\begin{prop}
Theorem \ref{thm:intro.der_mutl_via_fiber}(\ref{thm:intro.der_mutl_via_fiber.2}) holds. Namely,
if $\rho$ is cuspidal then $\cG_{V}(M,{\rho})\cong\cF_{V}((M,\rho))$.
\end{prop}
\begin{proof}
By definition we have
$$\cF_{V}((M,\rho))=\Hom_{G}(\Psi_{G}(M,\rho),V)=\Hom_{G}(\bar i_M^G(ind_{M_0}^M(\rho|_{M_0})),V).$$
Applying the second adjointness theorem we get
$$\Hom_{G}(\bar i_M^G(ind_{M_0}^M(\rho|_{M_0})),V)\cong \Hom_{M}(ind_{M_0}^M(\rho|_{M_0}),r_{M}^{G} V).$$

By \EitanG{Corollary} \ref{lem: Homs} we get 
$$\cF_{V}((M,\rho)) \cong \Hom_{M}(ind_{M_0}^M(\rho|_{M_0}),\Sc(M)) \otimes_{{\cH(M)}}  r_{M}^{G}(V).$$

By Lemma \ref{lem: cusp. gen.}
We obtain 
$$\cF_{V}((M,\rho))  \cong 
ind_{M_0}^M(\tilde{\rho}|_{M_0}) \otimes_{{\cH(M)}}  r_{M}^{G}(V)$$


Finally, since $\Psi_{M}(M,\rho)$  is a projective object, we obtain

$$
\cG_{V}(M,{\rho}) \cong \Psi_{M}(M,\tilde{\rho})  \otimes^{L}_{{\cH(M)} } r_{M}^{G} (V) \cong \Psi_{M}(M,\tilde{\rho})  \otimes_{{\cH(M)} } r_{M}^{G} (V)  = \cF_{V}((M,\rho))
$$
\end{proof}


\begin{prop}
Theorem \ref{thm:intro.der_mutl_via_fiber}(\ref{thm:intro.der_mutl_via_fiber.1}) holds. Namely,
let $V \in \cM(G)$ 
Then,

$$\RHom_G(V,i_M^G(\chi \tau))\cong \RHom_{O(\RamiD{\CH_M})}(\cG_{V}(M,\tau), \delta_{\chi}),$$

 where $\delta_{\chi}$ is the skyscraper sheaf over $\chi \in \RamiD{\CH_M}.$
\end{prop}


\begin{proof}
$ $
\begin{enumerate}[Step 1]
\item Proof for the case $M=G.$\\
By Corollary \ref{cor:tens.mult} it is enough to prove that 
$$V \otimes^{L}_{\cH(G)} i_M^G(\chi \tau)) \cong \cG_{V}(M,\tau) \otimes_{{O(\RamiD{\CH_M})}} \delta_{\chi}$$
This follows from associativity of tensor product.

\item Proof the general case.\\
By the previous step we have
$$\RHom_{O(\CH_M)}(\cG_{V}(M,\tau), \delta_{\chi}))\cong \RHom_M(r_M^G(V),\chi \tau)$$
By Frobenius reciprocity we get 
$$\RHom_M(r_M^G(V),\chi \tau)\cong \RHom_G(V,i_M^G(\chi \tau))$$

\end{enumerate}
\end{proof}

\begin{prop}\label{prop:perfect}
The following generalization of Theorem \ref{thm:intro.der_mutl_via_fiber}(\ref{thm:intro.der_mutl_via_fiber.3}) holds. Namely,
let $V \in \cM(G)$ 
be such that for any compact open $K<G$ the module  $V^K$ is finitely generated over $\cH(G,K)$.  
Then, $\cG_{V}(M,\tau)$ is perfect.
\end{prop}
For the proposition we will need the following \EitanG{lemmas}

\begin{lem}\label{JacIsFg}
 Let $V \in \cM(G)$ be a locally finitely generated module.
Then for any Levi group $M$, the representation $r_{M}^G(V)$ is locally finitely generated.
\end{lem}
\begin{proof}
First note that by Iwasawa decomposition (see e.g. page 40 \cite{BerLec}) if $V$ is finitely generated over $G$ then it is finitely generated over $P$. Thus if $V$ is finitely generated over $G$ then $r_{M}^G(V)$ is finitely generated over $M$.

Now, let $V \in \cM(G)$ be a locally finitely generated module, and let $K'\subset M$ be an open compact subgroup. 
For any open compact subgroup $K' < M$, one can find open compact subgroup  $K''<G$ s.t. $K''\cap M \subset K'$. By Bruhat theorem (page 41 \cite{BerLec}) one can find  open compact subgroup  $K < K''$ satisfying the conditions of Jacquet's lemma (page 65 in loc. cit.). Thus the map
$$V^K\to r_{M}^G(V)^{K\cap M}$$ is onto. 

This implies that  $$r_{M}^G(G \cdot V^K) \supset r_{M}^G(M \cdot V^K)= M\cdot r_{M}^G(V)^{K\cap M}\supset M\cdot r_{M}^G(V)^{K'}.$$ We know that $G \cdot V^K$ is finitely generated over $G$. Thus $r_{M}^G(G \cdot V^K)$ finitely generated over $M$. By Noetherity, this implies that $M\cdot r_{M}^G(V)^{K'}$ is finitely generated over $M$, \EitanG{in other words $r_{M}^G(V)^{K'}$ is finitely generated over $\cH_{K'}(M)$}.

\end{proof}
\begin{lemma}\label{lem:as.alg} Let $B$ a unital associative algebra and $C$ be a commutative algebra. Assume that $B \otimes C$ have finite homological dimension.
Let $V$ be finitely generated module over $B \otimes C$ and  $T$ be finite dimensional right module over $B \otimes C$. 
Then $\Tor_B^*(T, V)$ is finitely generated over $C$ with respect to the diagonal action.
\end{lemma}

\begin{proof}
$ $
\begin{enumerate}[{Case} 1]
\item $V=B \otimes C$  and the action of \EitanG{$C$ on $T$ is  given by a character.}\\
This is Obvious.
\item $V$ is free module\\
Follows from the previoues step.
\item $V$ is projective\\
Follows from the previous step.
\item The general case\\
Follows from the previous step, by induction on the homological dimension of $V$.
\end{enumerate}

\end{proof}

The following lemma is straightforward. 
\begin{lemma}\label{lem:ten.K} 
Let $\tau,V\in \cM(G)$ and let  $K<G$ be an open compact subgroup.  Assume that $\tau$ is generated by $\tau^K$. Then $$\tau \otimes_{\cH(G)} V\cong \tau^K \otimes_{\cH_K(G)} V^K$$
 \end{lemma}
%

We thus obtain:
\begin{cor}\label{cor:ten.K} 
Let $\tau,V\in \cM(G)$ and let  $K<G$ be an splitting open compact subgroup.  Assume that $\tau$ is generated by $\tau^K$. Then $$\tau \otimes^L_{\cH(G)} V\cong \tau^K \otimes^L_{\cH_K(G)} V^K$$
 \end{cor}

\begin{lemma}\label{lem:M1} 
Let $G_1=G_0\cdot Z(G)$ where $G_{0}$ is the subgroup of $G$ generated by compact subgroups, $Z(G)$ the center of $G.$ Let $\Lambda_{Z(G)}=\fX_{*}(Z(G)).$ Our choise of a uniformizer allows us to identify $\Lambda_{Z(G)}$ with a subgroup of $G.$
Then $\nu: \EitanG{G_{0}} \times \Lambda_{Z(G)} \to G_{1}$ is an isomorphism.  In particular $\cH(G_1)=\cH(G_0) \otimes \C[\Lambda_{Z(G)}]$
 \end{lemma}

\begin{proof}
By \cite[page 86]{BerLec}  the map $Z(G)/Z(G)^{0} \to G/G^{0}$ is injective and hence  $Z(G)^{0}=Z(G) \cap G^{0}.$ Thus $Ker(\nu) \cong G_{0} \cap \Lambda=Z_{0} \cap \Lambda=\{1\}$ and $Im(\nu)=G_{0} \cdot \Lambda =G_{0} \cdot Z(G)_{0} \cdot \Lambda=G_{0} \cdot Z(G)=G_{1}.$  
 \end{proof}

\begin{proof}[Proof of Proposition \ref{prop:perfect}]
\EitanG{Since $\cM(G)$ has finit homological dimension it is enough to prove that $H^*(\cG_{V}(M,\tau))$ is finitly generated.}
\begin{enumerate}[Step 1]
\item Proof for the case $G=M$\\
Since $G_{1}$ is of finite index in $G$ we get that $V$ is  localy finitely generated over $G_{1}.$  Chose splitting open compact $K<G$ such that $\tau$ is generated by $\tau^K$. 

By Corollary \ref{cor:ten.K}  we have:
$$\cG_{V}(G,\tau):=ind_{G_0}^G (\tau|_{G_0}) \otimes^L_{\cH(G)} V=(ind_{G_0}^G (\tau|_{G_0}))^K \otimes^L_{\cH_K(G)} V^K.$$

Now we have:
\begin{multline*}
(ind_{G_0}^G (\tau|_{G_0}))^K \otimes^L_{\cH_K(G)} V^K= (\tau^K   \otimes_{\cH_K(G_0)}  \cH_K(G)) \otimes^L_{\cH_K(G)} V^K =\\
=\tau^K   \otimes^L_{\cH_K(G_0)}  \cH_K(G) \otimes^L_{\cH_K(G)} V^K =\tau^K   \otimes^L_{\cH_K(G_0)}  V^K
\end{multline*}

This implies (by lemmas \ref{lem:as.alg} and \ref{lem:M1}) that \EitanG{$H^*(\cG_{V}(M,\tau))$} is finitely generated over \EitanG{$\C[G^1/G^0]$}, and hence over $G/G^0$
\item Proof for the general case\\
Follows from the previous step and Lemma \ref{JacIsFg}.
\end{enumerate}





\end{proof}


\begin{thebibliography}{MMMM}
\bibitem[\href{http://arxiv.org/abs/0910.3199v1}{AAG11}]{AAG} A. Aizenbud,  N. Avni and D. Gourevitch, {\it
Spherical pairs over close fields},
\RamiG{Commentarii Mathematici Helvetici, 87/4 (2012)}. See also arXiv:0910.3199[math.RT].

\bibitem[\href{http://projecteuclid.org/DPubS?service=UI&version=1.0&verb=Display&handle=euclid.dmj/1251120011}{AG09a}]{AG_HC}
A. Aizenbud, D. Gourevitch, {\it Generalized Harish-Chandra descent, Gelfand pairs and an
Archimedean analog of Jacquet-Rallis' Theorem.} Duke Mathematical Journal, 149/3 (2009). See also arxiv:0812.5063[math.RT].

\bibitem[\href{http://arxiv.org/abs/0808.2729}{AG09b}]{AG_AMOT} A. Aizenbud,  D. Gourevitch,
{\it Multiplicity one theorem for $(GL_{n+1}(\bR),GL_{n}(\bR))$}, Selecta Mathematica, 15/2,
(2009).
See also arXiv:0808.2729 [math.RT].

\bibitem[\href{http://arxiv.org/abs/0709.4215}{AGRS10}]{AGRS} A. Aizenbud,  D. Gourevitch, S. Rallis, G. Schiffmann, {\it
 Multiplicity One Theorems}, Annals of Mathematics  172/2 (2010),see also  arXiv:0709.4215 [math.RT].
 \bibitem[\href{http://arxiv.org/abs/0709.1273}{AGS08}]{AGS} A. Aizenbud,  D. Gourevitch, E. Sayag : {\it
 $(GL_{n+1}(F),GL_n(F))$ is a Gelfand pair for any local field
 $F$}, 
Compositio Mathematica, 144
(2008),
see postprint: arXiv:0709.1273[math.RT]





\bibitem[\href{}{AGS15}]{AGS_Z} A. Aizenbud,  D. Gourevitch, E. Sayag : {\it
Z-finite distributions on p-adic groups}. Advances in Mathematics 285/5 (2015), see also ArXiv: 1405.2540.



\bibitem[\href{http://www.math.tau.ac.il/~bernstei/Publication_list/publication_texts/Bernstein87-second-adj-from-chicago.pdf}{Ber87}]{BerSA}
J.N. Bernstein, \emph{Second Adjointness Theorem for Representations of p-adic Groups}, 1987.


\bibitem[\href{http://www.math.tau.ac.il/~bernstei/Publication_list/publication_texts/Bern_Center.pdf}{BD84}]{BD}
J.N. Bernstein, \emph{Le centre de Bernstein (edited by P.
Deligne)} In: Representations des groupes reductifs sur un corps
local, Paris, 1984, pp. 1-32.



\bibitem[\href{http://www.math.tau.ac.il/~bernstei/Publication_list/publication_texts/Bernst_Lecture_p-adic_repr.pdf}{BR}]{BerLec}
J.N. Bernstein and K. Rumelhart, \emph{Lectures on $p$-adic groups}. Unpublished, available at
\url{http://www.math.tau.ac.il/~bernstei/Publication_list/}.

\bibitem[BZ76]{BZ} J. Bernstein, A.  Zelevinsky, {\it Representations of $GL(n,F)$ where $F$ is a non-Archimedean local field}, Russian Mathematical Surveys, 31:3 (1976), 1--68.
\bibitem[Cas]{Cas} W. Casselman, {\it A new nonunitarity argument for $p$-adic representations}, J. Fac. Sci. Univ. Tokyo Sect. IA Math. 28 (1981), no. 3, 907-928, MR0656064 (1984e:22018).


\bibitem[CD94]{CD}
 J. Carmona and P. Delorme. {\it Base m\'eromorphe de vecteurs distributions H-invariants pour les s\'eries principales g\'en\'eralis\'ees d\'espaces sym\'etriques r\'eductifs: equation fonctionnelle.} J. Funct. Anal., 122/1, (1994).

\bibitem[\href{http://www.ams.org/journals/tran/2010-362-02/home.html}{Del10}]{Del} P. Delorme,
 \emph{Constant term of smooth $H_\psi$-spherical functions on a reductive $p$-adic group}. Trans. Amer. Math. Soc. \textbf{362}  (2010), 933-955. See also
\url{http://iml.univ-mrs.fr/editions/publi2009/files/delorme_fTAMS.pdf}.

\bibitem[\href{http://arxiv.org/abs/0909.2999}{GGP12}]{GGP} W. T. Gan, B. Gross, D. Prasad, {\it Symplectic local root numbers, central critical L-values and restriction problems in the representation theory of classical groups}, Asterisque 346 (2012), arxiv:0909.2999.




\bibitem[GK75]{GK} I.M. Gelfand, D. Kazhdan, {\it Representations of the group ${\rm GL}(n,K)$ where $K$ is a local
 field}, Lie groups and their representations (Proc. Summer School,
 Bolyai Janos Math. Soc., Budapest, 1971), pp. 95--118. Halsted,
 New York (1975).

\bibitem[GSS15]{GSS}
D. Gourevitch, S. Sahi, E. Sayag, {\it Invariant functionals on Speh representations},
Transform. Groups,
Transformation Groups, 20/4 (2015)



\bibitem[\href{http://arxiv.org/abs/0709.3506}{HM08}]{HM}
 J. Hakim, F. Murnaghan
\emph{Distinguished tame supercuspidal representations}, Int. Math. Res. Pap. IMRP 2008, no. 2,
see also arxiv0709.3506.

\bibitem[\href{http://archive.numdam.org/ARCHIVE/CM/CM_1996__102_1/CM_1996__102_1_65_0/CM_1996__102_1_65_0.pdf}{JR96}]{JR}
H. Jacquet, S. Rallis, {\it Uniqueness of linear periods.}
Compositio Mathematica ,  102/1,
(1996).


\bibitem[Len]{Len}
H. Lenzing, {\it Endlich prasentierbare Moduln}, 
Arch. Math. (Basel) 20 (1969), 262–266

\bibitem[OS08a]{OS_e}
O. Omer,  S. Eitan
{\it Global mixed periods and local {K}lyachko models for the
              general linear group}, Int. Math. Res. Not. IMRN. 2008/1

\bibitem[\href{http://arxiv.org/abs/0711.2884}{OS08b}]{OS} O. Offen, E. Sayag {\it Uniqueness and disjointness of Klyachko models}, Journal of Functional Analysis
254/11, (2008), see also Arxiv:0711.2884.

 \bibitem[Pra]{Pra} D. Prasad, {\it Ext-analogues of Branching laws}, See arXiv:1306.2729[math.RT].

\bibitem[Pra90]{Pra_m}
D. Prasad,
\emph{Trilinear forms for representations of $GL(2)$ and local $\eps$-factors.} Compositio Math. 75 (1990), no. 1.
\bibitem[SV]{SV} Y. Sakellaridis and A. Venkatesh, \emph{Periods and harmonic analysis on spherical varieties}.  ArXiv:1203.0039.




%




\bibitem[Sha74]{Sha} J.A. Shalika, Multiplicity one theorem for GLn, Ann. Math. \textbf{100} (1974), 171-193.

\bibitem[\href{http://arxiv.org/abs/0903.1413}{SZ12}]{SZ} B. Sun and C.-B. Zhu, {\it
Multiplicity one theorems: the Archimedean case}, Annals of Mathematics 175/1 (2012)
arXiv:0903.1413[math.RT].

\bibitem[vB88]{vB}
van den Ban, E. P.,
      {\it The principal series for a reductive symmetric space. {I}.
              {$H$}-fixed distribution vectors},
Ann. Sci. \'Ecole Norm. Sup. (4),
    21/3
     (1988)


\bibitem[vD08]{vD} G. van Dijk,  {\it Gelfand pairs and beyond}. COE Lecture Note, 11. Math-for-Industry Lecture Note Series. Kyushu University, Faculty of Mathematics, Fukuoka, 2008. ii+60 pp.









\end{thebibliography}
\end{document}